\documentclass[10pt, conference]{ieeeconf}      
\IEEEoverridecommandlockouts                              
\usepackage{cite}\usepackage{hyperref}
\usepackage{amsmath,amssymb,amsfonts}
\usepackage{algorithmic}
\usepackage{graphicx}
\usepackage{textcomp}
\usepackage{amsmath,amssymb,bm,bbm,mathrsfs,amscd}
\usepackage{calc}
\usepackage{color}
\usepackage{dsfont}
\usepackage{graphicx}
\usepackage{epstopdf}
\usepackage{epsfig}
\usepackage{tikz}
\usepackage{amsmath}
\usepackage{amsfonts}
\usepackage{amssymb}
\usepackage{rotating}
\usepackage{mathtools}
\usepackage{color}
\usepackage{subfig}
\usepackage{enumerate,pgfplots}
\newtheorem{theorem}{Theorem}

\newtheorem{proposition}{Proposition}

\newtheorem{remark}{Remark}

\newcommand{\ba}{\begin{array}}
	\newcommand{\ea}{\end{array}}

\newcommand{\be}{\begin{equation}}
\newcommand{\ee}{\end{equation}}

\newcommand{\mc}{\mathcal}

\def\1{\boldsymbol{1}}

\newcommand{\R}{\mathbb{R}}

\newcommand{\summ}{\sum\limits}

\def\R{\mathbb{R}}

\def\diag{{\rm diag}\,}
\tikzstyle{v_c}=[circle, draw,inner sep=2pt, minimum width=12pt, color=blue]
\tikzstyle{v_a}=[circle, draw,inner sep=2pt, minimum width=12pt, color=red]
\tikzstyle{edge} = [draw,thick,-,font=\small ]
\tikzstyle{label} = [draw,fill=black,font=\normalsize]

\def\BibTeX{{\rm B\kern-.05em{\sc i\kern-.025em b}\kern-.08em
		T\kern-.1667em\lower.7ex\hbox{E}\kern-.125emX}}
\title{\LARGE \bf On SIR epidemic models with feedback-controlled interactions and network effects}
\author{Martina~Alutto, 
	Giacomo~Como,~\IEEEmembership{Member,~IEEE,}
	and Fabio Fagnani
	\thanks{The authors are with the  Department of Mathematical Sciences ``G.L.~Lagrange,'' Politecnico di Torino, 10129 Torino, Italy  (e-mail: {\{martina.alutto;\,giacomo.como;\,fabio.fagnani;\}@polito.it}). G.~Como is also with the Department of Automatic Control, Lund University, 22100 Lund, Sweden.}
	\thanks{This work was partially supported by a MIUR grant ``Dipartimenti di Eccellenza 2018--2022'' [CUP: E11G18000350001], a MIUR  Research Project PRIN 2017 ``Advanced Network Control of Future Smart Grids'' (http://vectors.dieti.unina.it) and by the Compagnia di San Paolo.}%
}

\begin{document}

	\maketitle
	\thispagestyle{empty}
	\pagestyle{empty}
	
\begin{abstract}	We study extensions of the classical SIR model of epidemic spread. First, we consider a single population modified SIR epidemics model in which the contact rate is allowed to be an arbitrary function of the fraction of susceptible and infected individuals. This allows one to model either the reaction of individuals to the information about the spread of the disease or the result of government restriction measures, imposed to limit social interactions and contain contagion. We study the effect of both smooth dependancies and discontinuities of the contact rate. In the first case, we prove the existence of a threshold phenomenon that generalizes the well-known dichotomy associated to the reproduction rate parameter in the classical SIR model. Then, we analyze discontinuous feedback terms using tools from sliding mode control. Finally, we consider network SIR models involving different subpopulations that interact on a contact graph and present some preliminary simulations of modified versions of the classic SIR network. 
		
\end{abstract}

\section{Introduction}\label{sec:introduction}
As a result of the COVID-19 pandemic, there has been a renewed interest in the mathematical modeling, analysis, and control of epidemic spreadings. See, e.g., \cite{Giordano.ea:2020,Birge.ea:2020,Alvarez.ea:2020,Acemoglu.ea:2020,Miclo.ea:2020,Zino.Cao:2021}. Of the two best known mathematical models of epidemics ---the so-called susceptible-infective-succeptive (SIS) and susceptible-infective-recovered (SIR) models--- it is the latter that better approximates the spread of diseases like COVID-19 over a time horizon during which individuals tend not to get infected more than once, either because they have deceased or since they have recovered achieving some degree of immunity. The deterministic SIR epidemic model, as first presented in the pioneering work \cite{Kermack.McKendrick:1927}, is a compartmental model consisting of a nonlinear system of three coupled differential equations describing the evolution of the fractions of susceptible, infected, and removed individuals in a fully mixed closed population.

The main feature of the deterministic SIR model is the existence of a phase transition described in terms of a scalar parameter, known as the reproduction number, whose value can determine two fundamentally different behaviors of the epidemics. Specifically, if the reproduction number does not exceed a unitary value, then the fraction of the infected individuals is bound to remain monotonically decreasing in time, and in fact asymptotically vanishing as time grows large, thus preventing an epidemic outbreak. In contrast, if the reproduction number exceeds the unitary threshold, then the fraction of infected individuals is initially increasing until reaching a peak, after which it starts to decrease monotonically as in the previous case and vanishes asymptotically in the large time limit. Thus, reproduction number values above one are equivalent to the occurrence of an epidemic outbreak. This phase transition is a crucial aspect of the epidemics and motivates the recent focus on a correct estimation of the reproduction number, as well the attempts to design control policies capable to stir the reproduction number below the unitary threshold level.

The original deterministic SIR model relies on the assumption that the rate at which individuals get infected is proportional to the product between the fraction of the susceptible individuals and the fraction of infected individuals. This is an appropriate model if we envision a fully mixed population with constant contact and transmission rates.  
There are many reasons for considering different interaction terms, as the result of either endogenous or exogenous interactions, e.g.: 
\begin{itemize}
\item self-isolation policies or precautionary measures such as mask wearing put into place by the individuals who have become aware of the danger of the epidemic spreading;
\item feedback policies enforcing contact rate reduction by a central controller;
\item heterogeneities in the population and network effects. 
\end{itemize}

In this paper, we consider epidemic models that address the points listed above. Specifically, we study two extensions of the classical SIR model. First, we introduce a modified deterministic SIR model with a general interaction term describing the contact frequency rate as an arbitrary function of the fractions of susceptible and infected individuals in the population. Such function is typically decreasing in the fraction of infected individuals and can be interpreted as an endogenous reaction to the spread of epidemics (people tend to self-isolate), or rather an exogenous feedback control term modeling the action of a central planner actuating some partial social distancing measure. Then, we consider deterministic network SIR models where nodes represent subpopulations, derived by a split by either biological attributes (age, gender, risk) or geographical ones. 


Our contribution is three-fold. First, for the extended scalar SIR model described above, we prove in Section \ref{sec:2} that as long as the contact rate is a smooth function of the state that is nondecreasing in the fraction of susceptible individuals, we retrieve the same threshold behavior of the SIR model: either the infection dies out  monotonically, or it first increases, reaches a peak, and then decreases monotonically until vanishing asymptotically. Notice that in both cases the fraction of infected individuals remains a unimodal function of time, i.e., it adimts an unique local maximum that occurs at $t=0$ when the the reproduction number does not exceed the unitary value and at $t>0$ if the reproduction number value is larger than $1$. Second, in Section \ref{sec:3}, we consider the generalized scalar SIR model where the contact rate is a discontinuous feedback term. For the especially important case of piecewise constant feedback controls, we show conditions under which a new sliding motion phenomenon can arise. In fact, we prove that in such models the fraction of infected individuals can remain constant at its maximum level, for a trivial interval of time before starting its monotone convergence to $0$. Finally, in Section \ref{sec:4} we study the deterministic network SIR model and show the effect of introducing two different types of control within the classical model. The former is a limitation of interactions between different nodes, while the latter is a general reduction of any kind of interaction.



\section{A SIR model with general interaction term}\label{sec:2}
We consider the following system of ODE's
\begin{equation}
	\label{eq:eq2}
	\begin{cases}
		\dot{x}(t) = -x(t) y(t) f(x(t),y(t))\\
		\dot{y}(t) =  x(t) y(t)f(x(t),y(t)) - \gamma y(t)\\
		\dot{z}(t) = \gamma y(t)\,, 
	\end{cases}
\end{equation}
where $f : \R^{2} \mapsto \mathbb{R}_{+}$ is a state-dependent contact rate and $\gamma>0$ is a constant recovery rate. Notice how, in the special case  $f(x,y)=\beta>0$, i.e., when the contact rate is a positive constant, \eqref{eq:eq2} reduces to the classical deterministic SIR model 
\begin{equation}
	\label{eq:eq0}
	\begin{cases}
		\dot{x}(t) = -\beta x(t) y(t) \\
		\dot{y}(t) =  \beta x(t) y(t) - \gamma y(t)\\
		\dot{z}(t) = \gamma y(t) \,,
	\end{cases}
\end{equation}
first introduced and studied in \cite{Kermack.McKendrick:1927}. 

For now we assume $f$ to be a $\mc C^1$ function. As it happens for the SIR model, the three equations are dependent, yielding that $x(t)+y(t)+z(t)$ is constant. Throughout the paper, we assume this constant to be one, so to interpret $x(t)$, $y(t)$, and $z(t)$ as fractions. The same considerations than in the SIR model, moreover, lead to the fact that solutions of (\ref{eq:eq2}) are globally defined (and unique) and that they keep invariant the simplex 
$$\mc S=\left\{(x,y,z)\in\R^3_+\,|\, x+y+z=1\right\}\,.$$
From now on we assume to always pick an initial condition that lays in $\mc S$. Considering that $x(t)$ and $z(t)$ are monotone functions of time, the former  decreasing and the latter increasing, we obtain that they always converge to a limit as $t$ grows large. Invariance of the simplex $\mc S$ implies that the all solution vectors converges asymptotically and that in the limit $y(t)\to 0$. 
Now, extending the contribution in \cite{Capasso1978AGO}, we prove that, under mild assumptions on the function $f$, the threshold behavior of the SIR is retrieved in this model.

To this aim, define, for every solution $(x(t), y(t), (t))$ of (\ref{eq:eq2}), the function
\be R(t)=x(t) f(x(t),y(t))/ \gamma\ee 
We now show that $R(t)$ plays the same role than the usual reproduction number for the SIR model (to which it reduces when $f$ is constant). The following invariant result holds.

%
%
%
%
%
%
%
%
\begin{proposition}
	\label{prop-invariant}
	Assume that $f $ is of class $\mc C^1$ and is such that $\partial f/\partial x \geq 0$ in every point of $\mc S$. 
			For every solution of the ODE (\ref{eq:eq2}), if $R(0)<1$, then $R(t)<1$ for all $t\ge0$.
\end{proposition}

\begin{proof} We first notice that if $x(0)=0$, then $x(t)=0$ at all time and consequently $R(t)=0$ at all time. We now consider the case when $x(0)>0$ that yields (by uniqueness of the solution) $x(t)>0$ at all time. In this case, we prove the result by contradiction. If not, by continuity, there exists $t^*>0$ such that $R(t^*)=1$ and $R(t)<1$ for all $t<t^*$.
We now compute the time derivative of $R(t)$:
\begin{equation}\label{dotR}
		\gamma \dot R = \dot{x} f + x \dfrac{\partial f}{\partial x}\dot{x} + x\dfrac{\partial f}{\partial y}\dot{y} 
	\end{equation}
We study the sign of $\dot{R}(t^*)$. Because of the sign of $\dot x$, the standing fact that $x>0$ at all time and the assumption on the function $f$, we have that the first two addends in the righthand side of \eqref{dotR} are always negative. Finally, the third term is $0$ at $t^*$ since $\dot y(t^*)=0$. We conclude that $\dot{R}(t^*)<0$. By continuity, we can state that $\dot{R}(t)<0$ in an interval $[t^{*} - \epsilon, t^{*}]$ for some $\epsilon >0$. Since $R(t)$ is decreasing in $[t^{*} - \epsilon, t^{*}]$ and stays strictly below $1$ for all $t<t^*$ for the assumption made, it follows that also $R(t^*)<1$ contrarily to what we had assumed. This yields the result.
%
%
%
%
%
%
%
%
\end{proof}\medskip

We can now state and prove the following result that shows how 
$R(t)$ plays the same exact role than the reproduction number for the SIR model.
\begin{theorem}
	\label{theo-SIRmod}
	Assume that $f $ is of class $\mc C^1$ and is such that $\partial f/\partial x \geq 0$ in every point of $\mc S$. Given an initial condition $(x_0, y_0, z_0)\in\mc S$ and called the corresponding solution of (\ref{eq:eq2}) as $(x(t), y(t), z(t))$, the following facts hold
				\begin{enumerate}
			\item[(i)] If $R(0)<1$, then $y(t)$ converges to $0$ monotonically.
		\item[(ii)] If $R(0)>1$ and $y(0)>0$, then there exists $t^*>0$ such that 
			\begin{itemize}
			\item $y(t)$ is monotonically increasing in $[0, t^*]$ 
			\item $y(t)$ is monotonically decreasing in $[t^*, +\infty[$ and converges to $0$
		\end{itemize}
	\end{enumerate}
\end{theorem}

\begin{proof}
Concerning (i), it follows from Proposition \ref{prop-invariant} that $R(t)<1$ for all $t$. This condition is equivalent to saying that $\dot{y}(t)<0$ at all times $t$ and proves (i).

Concerning (ii), notice that there must exist $t>0$ such that $R(t)<1$. Indeed, if not, $\dot{y}(t)>0 \, \forall t$ and $y(t)$ would not possibly tend to 0. If we now define
$$t^*=\inf\{t>0\,|\, R(t)<1\}$$
we have that by construction the properties expressed in (ii) hold true. The proof is now complete. 
\end{proof}

\begin{remark}\label{remark:classical} In the case when $f(x,y)=\beta$ is constant, we are back in the classical SIR model and Theorem \ref{theo-SIRmod} retrieves, in this case, the well known result on the behavior of curve of infected in this model. More details on the solutions can be obtained in this case. We briefly recall them below, as they will be needed in the next section. It will be convenient to set up the notation $\rho=\gamma/\beta$.

From the first and third equation in (\ref{eq:eq2}), we notice that the function
$$\Gamma(x,y,z) = \rho\ln x +z=\rho\ln x-x+1-y$$
is motion invariant. In particular, given the initial condition $x_0=1-\epsilon$,  $y_0=\epsilon$, and $z_0=0$, we obtain that the solution will lay in the manifold
\be\label{orbit}y=1-x+\rho\ln {x} -\rho\ln {(1-\epsilon)}\,.\ee
When $R(0)=(1-\epsilon)/\rho>1$, the maximum value reached by the infection happens in correspondence of $x=\rho$, as for this term $\dot{y}=0$, and is thus given by
\be\label{ymax}M(\epsilon, \rho)=1-\rho+\rho\ln \rho- \rho \ln (1-\epsilon)\,.\ee

\end{remark}

\bigskip
In most cases of interest, the feedback term $f$ is only function of $y$. Indeed, it is natural to imagine that a reaction both endogenous or exogenous be correlated to the extent of the current level of infection in the population. We notice that our result do not put any constraint on the way $f$ may depend on $y$. Natural feedback terms will however be decreasing in $y$. In Figure \ref{fig1} we compare the evolution of the classical SIR model in an unstable case with two modified versions having $f(x,y)=h^{2}(y)$ with respectively $h(y) = 1-y$ and $h(y)=(1-y)^2$.
\begin{figure}
	\subfloat[][]{
		\includegraphics[scale=0.185]{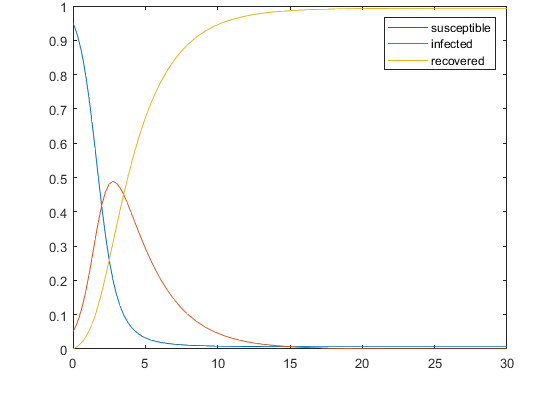}
	}
	\subfloat[][]{
		\includegraphics[scale=0.185]{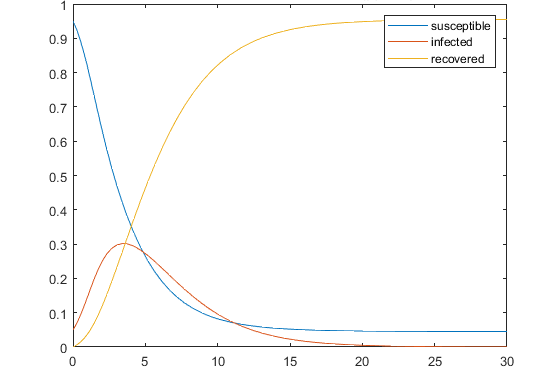}
	}
	\subfloat[][]{
		\includegraphics[scale=0.185]{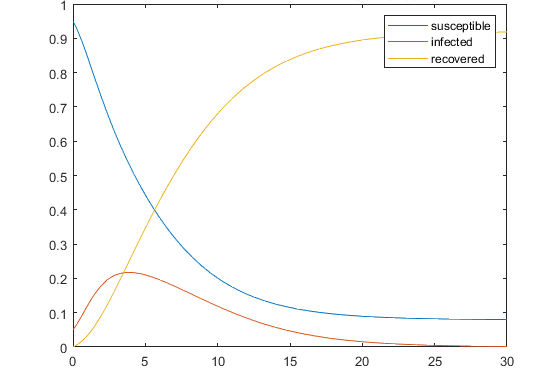}
	}
	\caption[Modified scalar SIR with linear and quadratic control with $R_{0}>1$.]{Simulations of the SIR model in the case of epidemic outbreak. The first figure on the left is the classical model, while the other two figures are the simulations of the modified SIR model with linear and quadratic interaction terms.}
\label{fig1}\end{figure}

\color{black}
\section{SIR model with threshold terms}\label{sec:3}
When we are modeling a control action of a central planner, it is of interest to study the case when $f(y)$ has discontinuities. Indeed, it is not feasible to imagine a policy that varies with continuity, rather it is more natural a policy that changes when the infection reaches certain thresholds. In this section, we study in detail the case when $f(y)$ is piecewise constant. In this case, the analysis carried on in previous section in general fails because of the discontinuities of the right hand side of the ODE (\ref{eq:eq2}). Classical solutions may not exist and, in this context, we will use the concept of solution according to Filippov \cite{filippov:88}.

We consider the ODE (\ref{eq:eq2}) with an interaction term $f(y)$ as defined below
\be\label{piecewise-constant} f(y)=\left\{\begin{array}{ll} \beta\quad &\hbox{if}\, y<k\\
\bar \beta\quad &\hbox{if}\, y\geq k\end{array}\right.\ee
We interpret $\beta$ as a sort of intrinsic interaction/contagion term that, in the absence of control measures, describes the rate at which infection propagates. When the infection gets above the threshold $k$, a (partial) lockdown policy takes place and brings this term to a smaller value $\bar \beta<\beta$. In accord to Remark \ref{remark:classical} we use the notation $\rho=\gamma/\beta$ and $\bar\rho=\gamma/\bar\beta$.

To analyze this model, it is convenient to focus on the first two equations of (\ref{eq:eq2}). The right hand side is a discontinuous vector field that present a so called sliding manifold:
$$\Omega=\left\{(x,y)\;|\; \rho\leq x\leq \min\left\{\bar\rho, 1\right\},\; y=k\right\}$$
Indeed, in a sufficiently small neighborhood of $\Omega$ the vector field of the ODE points in the direction of the manifold. 

This is illustrated in Figure \ref{fig2}.
Trajectories in that region will eventually hit the manifold $\Omega$ and then they will remain on it sliding in the direction of decreasing $x$ till the point  $(\rho, \bar y)$. From that point, the trajectory will remain in the region below (uncontrolled) and it will coincide with the trajectory of the uncontrolled SIR with interaction rate term equal to $\beta$.  
\begin{figure}
	\centering
	\subfloat[][]{
		\includegraphics[scale=0.1]{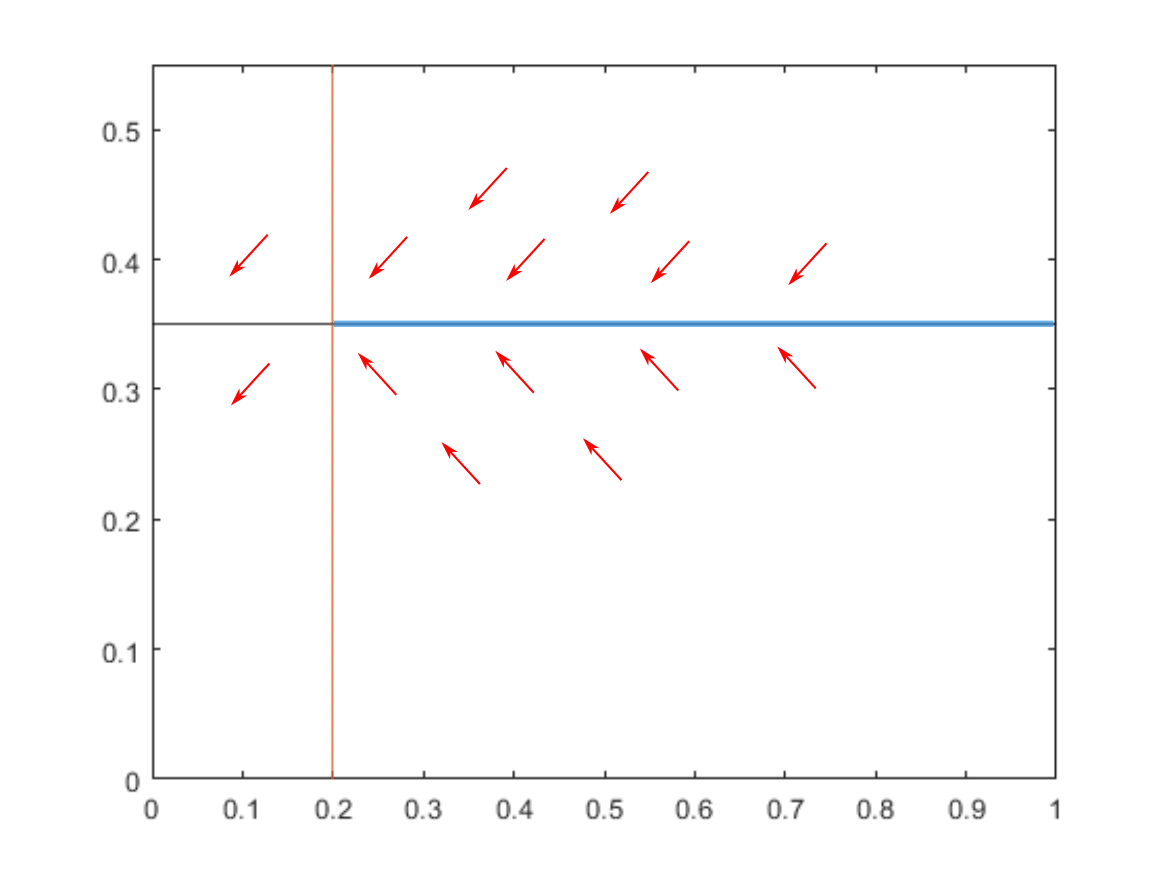}
	}
	\subfloat[][]{
		\includegraphics[scale=0.1]{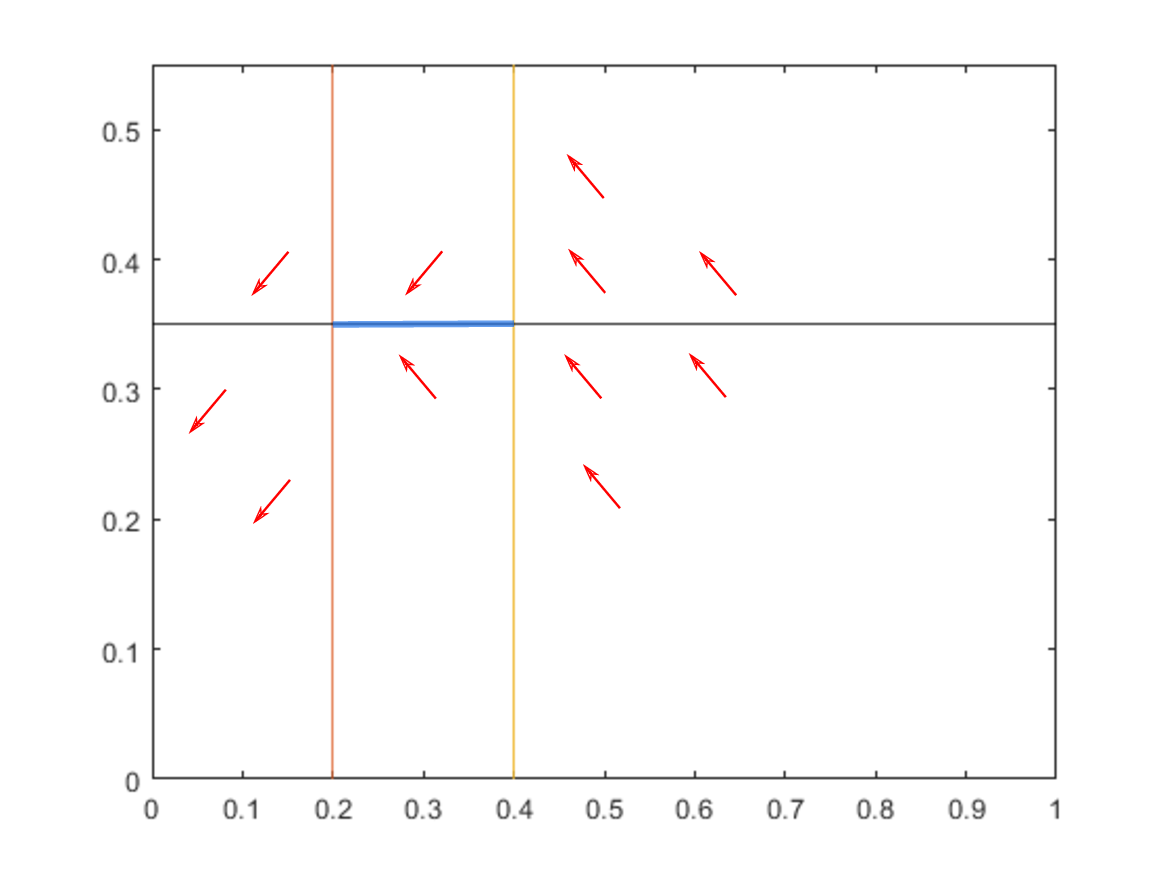}
	}
	\caption[]{Vector fields the ODE (\ref{eq:eq2}) with a piecewise constant feedback term in the $xy$ plane, for the case $\beta = 2$, $\gamma = 0.4$, $k=0.35$, and $\bar \beta = 0.38$ (left) and $\bar \beta = 1$ (right).}
\label{fig2}\end{figure}

In the following result we gather more detailed information on the nature of the sliding phenomenon and the conditions on the parameters for it to happen. To distinguish the two regimes determined by the value of the interaction term, we refer to, respectively, the $\beta$ and the $\bar \beta$-SIR model.

\begin{theorem} 
\label{theo-SIRdisc} 
	Assume that $f $ is as defined in (\ref{piecewise-constant}) and that the initial condition is of type $(x(0), y(0), z(0))=(1-\epsilon, \epsilon, 0)$.  Assume that  $(1-\epsilon)/\rho>1$ and put
	\be m(\epsilon, \rho, \bar\rho)=\left\{\begin{array}{ll}1-\bar\rho+\rho\ln \frac{\bar\rho}{1-\epsilon}\quad &\hbox{if}\, \rho<\bar\rho<1-\epsilon \\
	\epsilon\quad &\hbox{if}\, 1-\epsilon<\bar\rho\end{array}\right.
	\ee
	Consider the solution $(x(t), y(t), z(t))$ of (\ref{eq:eq2}) in the sense of Filippov. Then, the following facts hold: 
\begin{enumerate}
\item[(a)] If $M(\epsilon, \rho)<k$, then $y(t)< k$ at all times and the solution $(x(t), y(t), z(t))$ coincides at all time with that of a $\beta$-SIR model. In particular; the maximum value reached by the infection is $y^{\max}=M(\epsilon, \rho)$.
\item[(b)] If $m(\epsilon, \rho, \bar\rho)<k<M(\epsilon, \rho)$, then there exist time instants $0<t^*<t^{**}$ such that 
\begin{itemize}
\item $y(t)=k$ for all $t\in [t^*, t^{**}]$
\item $y(t)$ is monotonically increasing in $[0,t^*]$, monotonically decreasing in $[t^{**}, +\infty[$ and such that $\lim_{t\to0}y(t)=0$.
\end{itemize}
\item [(c)] If $\epsilon\leq  k\leq m(\epsilon, \rho, \bar\rho)$, then there exists a time instant $t^*>0$ such that 
			\begin{itemize}
			\item $y(t)$ is monotonically increasing in $[0, t^*]$; 
			\item $y(t)$ is monotonically decreasing in $[t^*, +\infty[$ and converges to $0$.
		\end{itemize}
		The maximum value in this third regime is given by the expression
	\be\label{ymax-contr}y^{max}=
		M(\epsilon, \bar\rho)+(\bar\rho-\rho)\ln\frac{1-\epsilon}{x(k)}
		\ee
		where $x(k)$ is the abscissa when the solution crosses the manifold $y=k$ and is explicitly described by the relation
	\be\label{ycross-contr}x(k)+\rho\ln {x(k)} -\rho\ln {(1-\epsilon)}+k-1=0\,.\ee
		\end{enumerate}
\end{theorem}\medskip

\begin{proof}
In the regime described in (a), the solution for the classical $\beta$-SIR model remains always in the region $y<\bar y$. Consequently, it is also a solution of the controlled SIR model with control (\ref{piecewise-constant}).

To study the other cases, we reduce to the $xy$ plane and we only consider the first two equations in (\ref{eq:eq2}). We denote by $(x(t), y(t))$ the solution of the uncontrolled $\beta$-SIR model.  We first note that, using the relation $(\ref{orbit})$, when $\bar\rho <1-\epsilon$ the expression $m(\epsilon, \rho, \bar\rho)$ coincides with the value $y$ of the solution corresponding to $x=\bar\rho$. In other terms, there exists $t_1\geq 0$ such that $x(t_1)=\bar\rho$ and $y(t_1)=m(\epsilon, \rho, \bar\rho)$. 

If we are in the regime described by (b), we notice that for sure $y(0)=\epsilon <k$ so initially the solution leaves in the (uncontrolled) region $y<k$. Indicate by $t^*>0$ the first time when the solution $(x(t), y(t))$ hits the threshold level $k$. Such an instant must exist since the maximum value reached by $y(t)$ is above $k$. When $\bar\rho <1-\epsilon$
notice that necessarily $t_1<t^*$ as  $y(t_1)<k$ and $y(t)$ is increasing till it reaches its maximum value. This implies that 
\begin{equation}\label{interval}\rho<x(t^*)<x(t_1)=\bar\rho\,.\end{equation}

When instead $\bar\rho >1-\epsilon$, we have that $x(t)<\bar\rho$ at all times $t$ so that (\ref{interval}) remains true.
This says that, in any case, the solution hits, at time $t^*$ the sliding manifold $\Omega$. Considering that the derivative of $x$ is always negative, according to the definition of Filippov solution, the solution of  the controlled SIR model from instant $t_1^*$ on will be sliding on $\Omega$ till it reaches the point  $(\rho, k)$. This is reached at some time $t^{**}>t^*$. From time $t^{**}$ on the solution coincides again with the solution of the classical $\beta$-SIR model and will be decreasing in the component $y$ and will converge to $0$.

Consider now the regime (c). The only interesting case is when  $\bar\rho<1-\epsilon$. Consider again $(x(t), y(t))$ the solution of the uncontrolled $\beta$-SIR model. By the considerations above, we have that at time $t_1>0$ when $x(t_1)=\bar\rho$ we have that $y(t_1)>k$. This implies that the solution $(x(t), y(t))$ has hit the line $y=k$ at some previous time $t^*$ for which $x(t^*)>\bar\rho$. 

As $(x(t^*), k)$ is out of the sliding manifold $\Omega$, the solution of the controlled SIR-model will continue with a just a jump in the first derivative and since then it will coincide with the solution of an unstable classical $\bar\beta$-SIR model. The component $y$ will reach a peak for $x=\bar\rho$ and will then decrease and hit again the line $y=k$ at some further time $t_2>t^*$. Depending on whether $x(t_2)<\rho$ or $x(t_2)>\rho$ the solution, in the first case, will undergo another jump in the first derivative and converge as a solution of a $\beta$-SIR model, while in the second case, will first slide along $\Omega$ to the point $(\rho, k)$ and then will converge again as a solution of a $\beta$-SIR model. 

Finally, the values for the maximum reached by the fraction of infected are simply obtained through the formula (\ref{ymax}) that computes the maximum value of the infected in classical SIR models.
\end{proof}\medskip

\begin{remark}
In the regime (c) of Theorem \ref{theo-SIRdisc}, it can happen that the solution, after reaching the peak, exhibits a sliding motion during the decreasing phase. We have not explicitly indicated this in the statement, as this phenomenon does not modify the maximum value reached globally by the component $y(t)$ of the solution. It was however noticed in the proof. In Figure \ref{fig3} we show all possible behaviors of the trajectory in the plane $xy$. 

Notice, moreover, that the expression (\ref{ymax-contr}) for the maximum value reached by the fraction of infected is composed of two term. The first one is the value it would reach under the assumption that $k=\epsilon$, namely that the controlled regime is active since the initial time. The true value $y^{max}$ is obtained from this adding a positive extra term that depends on $k$ and accounts for the fact that for a while the epidemics has growth with no control. The estimation of this term can be relevant in the decision process of policy to adopt.
\end{remark}

\begin{figure}
	\centering
	\subfloat[][]{\hspace{-.5cm}
		\includegraphics[scale=0.33]{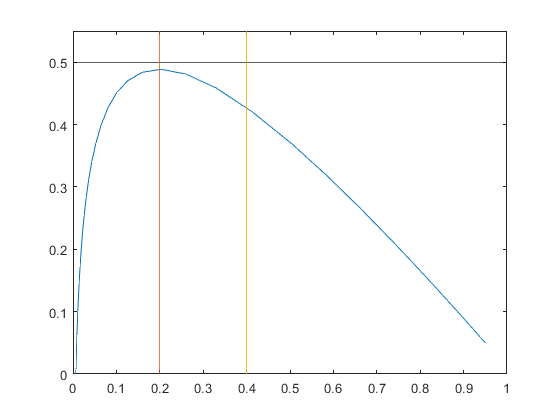}
	}
	\subfloat[][]{\hspace{-.5cm}
		\includegraphics[scale=0.33]{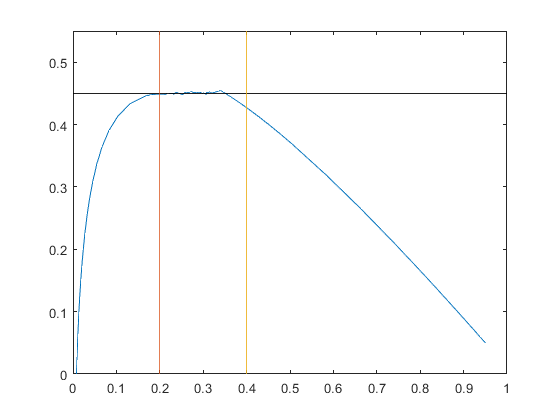}
	}\\
	\subfloat[][]{\hspace{-.5cm}
		\includegraphics[scale=0.33]{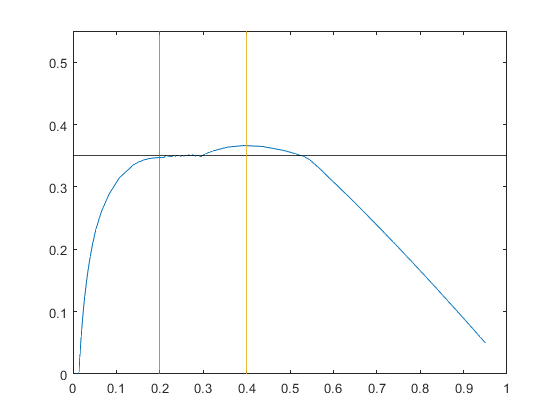}
	}
	\subfloat[][]{\hspace{-.5cm}
		\includegraphics[scale=0.33]{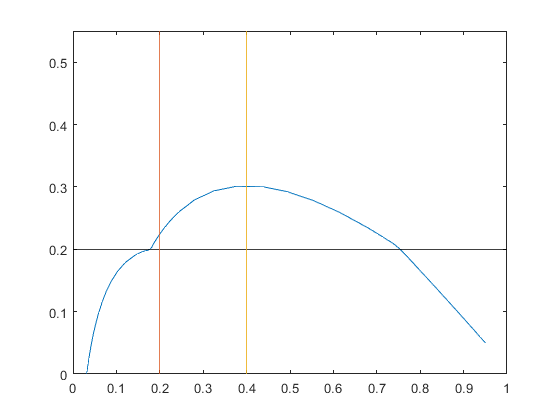}
	}
	\caption[]{All possible trajectory behaviors. Plot (a) corresponds to the first regime exposed in Theorem \ref{theo-SIRdisc}, plot (b) to the second regime, and, finally, plots (c) and (d) to the two possible situations of the third regime.}
	\label{fig3}
\end{figure}

\section{Network effects}\label{sec:4}
In this section, we analyse versions of the network SIR model.

Let $\mc G=(\mc V,\mc E,A)$ be a weighted di-graph with finite set of nodes $\mc V=\{1,2,\ldots,n\}$, set of directed links $\mc E\subseteq\mc V\times\mc V$, and adjacency/weight matrix $A$ in $\R_+^{n\times n}$. The different nodes $i$ in $\mc V$ represent different subpopulations whereas the positive entries $A_{ij}>0$ of the weight matrix are in one-to-one correspondence with the links $(i,j)$ in $\mc E$ and measure the contact frequency of members of subpopulation $i$ with members of subpopulation $j$. Throughout, we shall assume that the diagonal of $A$ is strictly positive, i.e., that $A_{ii}>0$ for all $i=1,\ldots,n$. 

For given infection rate $\beta>0$ and recovery rate $\gamma>0$, the network SIR epidemic model on a graph $\mc G=(\mc V,\mc E,A)$ is the dynamical system 
\begin{equation}
	\label{eq:network-SIR}
	\begin{cases}
		\dot{x}_i(t) = -\beta x_i(t)\sum_jA_{ij} y_j(t)\\
		\dot{y}_i(t) =  \beta x_i(t)\sum_jA_{ij} y_j(t) - \gamma y_i(t)\\
		\dot{z}_i(t) = \gamma y_i(t) \\
	\end{cases}
\end{equation}
for $i=1,\ldots,n$, where $x_i$, $y_i$, and $z_i$ represent respectively the fractions of susceptible, infected, and recovered individuals in population $i$. Notice that \eqref{eq:network-SIR} may be more compactly rewritten as 
\be\label{eqSIR-compact}\dot x=\beta\diag(x)Ay-\gamma y\,,\ \dot y=\beta\diag(x)Ay-\gamma y\,,\ \dot z=\gamma y\,.\ee

This model has been studied in \cite{Mei.ea:2017} and \cite{Nowzari.ea:2016}. In particular, it is known that all solutions converge to an the equilibrium point of the form $(x^*,0,z^*)$ in $\R_+^{3n}$ such that $x^*+z^*=\1$ and  that the locally asymptotically stable equilibrium points are those such that  
$$\lambda_{\max}(\diag(x^*)A)<\gamma/\beta\,,$$ 
where $\lambda_{\max}(M)$ stands for the dominant eigenvalue of a nonnegative matrix, which coincides with its spectral radius thanks to the Perron-Frobenius Theorem. 
In fact, under the assumption that the graph $\mc G$ is strongly connected, \cite[Theorem 7]{Mei.ea:2017} shows that the quantity 
$$R(t)=\frac{\beta}{\gamma}\lambda_{\max}(\diag(x(t))A)$$ is decreasing along solutions and it plays a role similar to the one played by the reproduction number in the scalar SIR model. Specifically,  if $R(0)\le1$ then the weighted average $v(0)'y(t)$ will monotonically decrease to $0$ as $t$ grows large; on the other hand if $R(0)>1$, then the weighted average $v(0)'y(t)$ will be initially increasing (epidemic outbreak) and there exists some $\tau>0$ such that $R(\tau)\le1$ and the weighted average $v(\tau)'y(t)$ will be decreasing to $0$ for $t$ in the interval $[\tau,+\infty)$. Here $v(t)$ stands for the nonnegative leading eigenvector of the matrix $\diag(x(t))A$.

Notice that the results summarized above concern the average behavior of the infection curve, with no implication on its behavior at individual nodes. 

We now study two versions of the SIR network model, in which we introduce a social interaction mitigation function and, as in the previous section, we assume it is dependent only on the fraction of infected and decreasing with respect to it. This control function can be introduced as a modification of interactions between different nodes or any type of interaction, both inter-nodal and within the same subpopulation.
If we consider nodes as geographically distinct subpopulations, the first type of contact limitation will lead to a kind of distancing and isolation per area with movements restriction. The interpretation of nodes in this model as a subdivision of the population into age groups will instead result in a limitation of interactions between people of different ages. This may be justified by an attempt to avoid contact between stronger people and people in age groups more vulnerable to disease, for example. We will then consider the following model

\begin{equation}
	\label{eq:network-SIR-mod}
	\begin{cases}
		\dot{x}_{i}(t) = -\beta x_{i}(t) \left( A_{i i}y_{i}(t) + \summ_{j\neq i}A_{i j}y_{j}(t)f_{i j}(t) \right) \\
		\dot{y}_{i}(t) = \beta x_{i}(t)\left( A_{i i}y_{i}(t) + \summ_{ j\neq i}A_{i j}y_{j}(t)f_{i j}(t) \right) - \gamma y_{i}(t)\\
		\dot{z}_{i}(t) = \gamma y_{i}(t)\,,
	\end{cases}
\end{equation}
where the term $f_{i j}(t)$ concerns the limitation of contacts between the individuals of population $i$ and those of population $j$. Obviously this term $f_{i j}(t)$ could depend on the infection level of both populations. 
For this reason, we can start by assuming that this term is equal to the product between the individual lockdown terms within the populations, that is
$$f_{i j} = f_{i} f_{j} \quad \forall i,j = 1,...,n $$
where the individual lockdown measures considered are the following functions
$$f_{i}(t) = 1-y_{i}(t) \quad \forall i = 1,..., n$$

A second alternative is to consider a control over all types of interaction within the network. In this case the studied model will be instead 
\begin{equation}
	\label{eq:network-SIR-mod2}
	\begin{cases}
		\dot{x}_{i}(t) = -\beta x_{i}(t) \left(A_{i i}y_{i}(t)f_{i}^2(t) + \summ_{j\neq i}A_{i j}y_{j}(t)f_{i j}(t) \right) \\
		\dot{y}_{i}(t) = \beta x_{i}(t)\!\left(\!\!A_{i i}y_{i}(t)f_{i}^2(t)\!+\!\summ_{j\neq i} A_{i j}y_{j}(t)f_{i j}(t)\!\!\right)\! - \gamma y_{i}(t)\\
		\dot{z}_{i}(t) = \gamma y_{i}(t) \,,
	\end{cases}
\end{equation}
where individual lockdown measures $f_{i}$ are assumed as in the previous case.

In Figure \ref{fig4} we show simulations of the network SIR model with $n=2$ subpopulations in the case of epidemic outbreak in both nodes and the two modified versions with the introduction of a internodal and a total control. Regarding the first node, these modifications of the model lead to an attenuation of the infection peak, while for the second node the introduction of the control causes the disappearance of an increasing trait for the curve of the infected.
\begin{figure}
	\subfloat[][]{
		\includegraphics[scale=0.28]{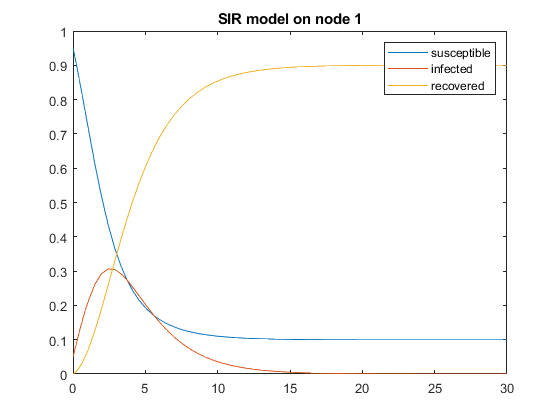}
	}
	\subfloat[][]{
		\includegraphics[scale=0.28]{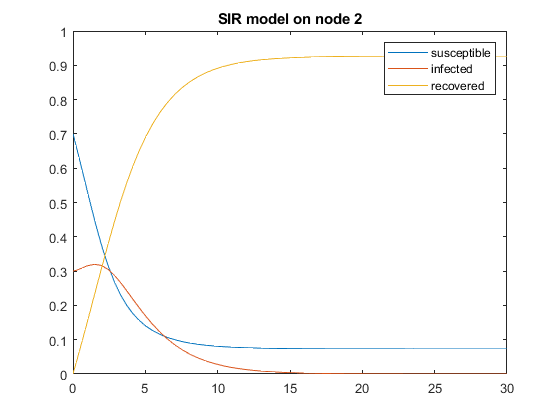}
	}\\
	\subfloat[][]{
		\includegraphics[scale=0.28]{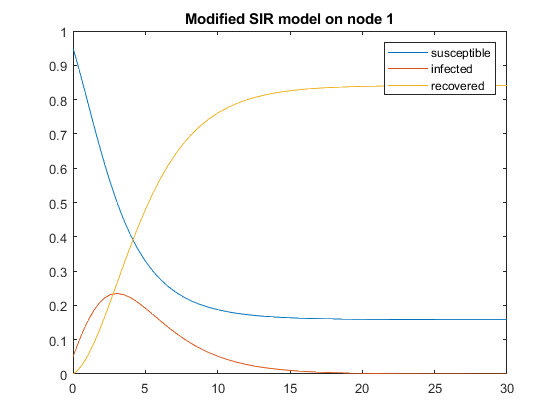}
	}
	\subfloat[][]{
		\includegraphics[scale=0.28]{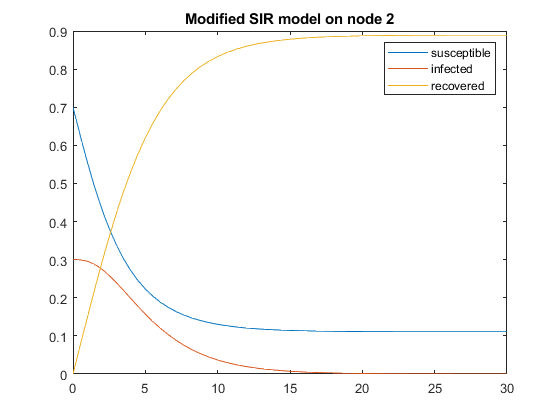}
	}\\
	\subfloat[][]{
		\includegraphics[scale=0.28]{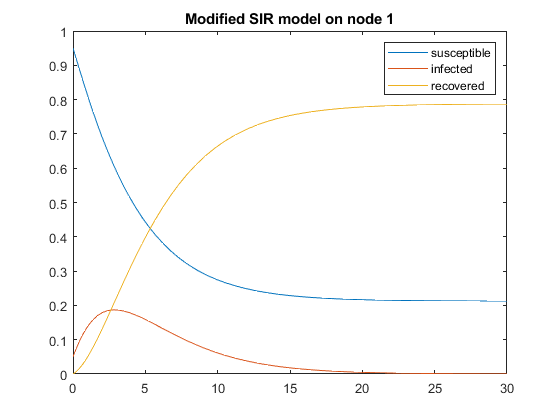}
	}
	\subfloat[][]{
		\includegraphics[scale=0.28]{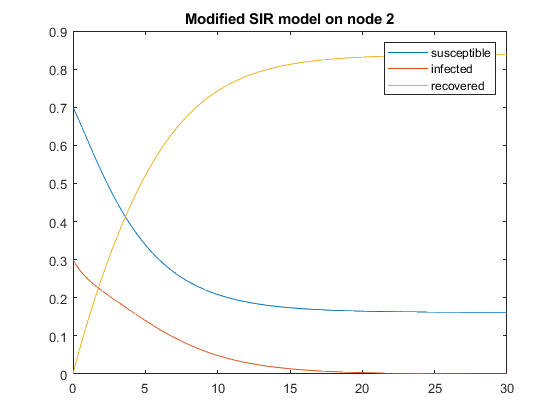}
	}
	\caption[]{Simulations of the SIR model in the case of epidemic outbreak. Plots (a) and (b) are the classical network model in each node, plots (c) and (d) are the modified SIR model with internodal control, while plots (e) and (f) are the simulations in both node with a total control.}
	\label{fig4}
\end{figure}

	\section{Conclusion}
We have studied extensions of the classical Kermack and McKendrick's SIR epidemic model \cite{Kermack.McKendrick:1927} that account for feedback-dependant contact rates and network effects. In particular, we have shown that discontinuous piecewise constant feedback rates may give rise to sliding motions, while for network models, simulations of possible modified versions are shown. Future research includes extension of these results, in particular for the network SIR model.

	\bibliographystyle{unsrt}
	\bibliography{bib}


\end{document}